\documentclass[11pt,a4paper]{article}
\usepackage[utf8]{inputenc}
\usepackage[english]{babel}
\usepackage{amsmath}
\usepackage{amsfonts}
\usepackage{amssymb}
\usepackage[left=3cm,right=3cm,top=3cm,bottom=3cm]{geometry} 

\usepackage{amssymb}
\usepackage{amsthm}

\usepackage{amsfonts}
\usepackage{amssymb}
\usepackage{indentfirst}
\usepackage{color,colortbl}
\usepackage[table]{xcolor}
\usepackage{stmaryrd}

\usepackage{array}

\usepackage{float}

\usepackage[pdftex]{graphicx}
\usepackage{caption}

\usepackage{mathrsfs}
\usepackage{tikz}
\newtheorem{theorem}{Theorem}[section]
\newtheorem{lem}[theorem]{Lemma}
\newtheorem{prop}[theorem]{Proposition}
\newtheorem{corollary}[theorem]{Corollary}

\newenvironment{notation}[1][Notation]{\begin{trivlist}
\item[\hskip \labelsep {\bfseries #1}]}{\end{trivlist}}

\usepackage{natbib}

\usepackage[titletoc]{appendix} 

\usepackage{hyperref} 
\hypersetup{
colorlinks=true, 
breaklinks=true, 
urlcolor= blue, 
linkcolor= blue, 
citecolor=blue, 
} 

\title{\textbf{A unified framework for the study of the PLS estimator's properties}}

\author{M\'elanie~Blaz\`ere,
        Fabrice~Gamboa
        and~ Jean-Michel~Loubes}

\date{} 

\begin{document}

\maketitle

\begin{abstract}
In this paper we propose a new approach to study the properties of the Partial Least Squares (PLS) estimator. This approach relies on the link between PLS and discrete orthogonal polynomials. Indeed many important PLS objects can be expressed in terms of some specific discrete orthogonal polynomials, called the residual polynomials. Based on the explicit analytical expression we have stated for these polynomials in terms of signal and noise, we provide a new framework for the study of PLS. Furthermore, we show that this new approach allows to simplify and retreive independent proofs of many classical results (proved earlier by different authors using various approaches and tools). This general and unifying approach also sheds new light on PLS and helps to gain insight on its properties. 
\end{abstract}

\begin{center}
{\bf \small Keywords} 
\end{center}
\begin{center}
Partial Least Square, multivariate regression, multicollinearity, dimension reduction, constrainsted least square, orthogonal polynomials, shrinkage.
\end{center}

\section{Introduction}
The PLS method, first introduced and developped by Wold in 1975, is an alternative to Ordinary Least Squares (OLS) when the explanatory variables are highly collinear or when they outnumber the observations. This method has been successfully applied in a wide variety of fields and has gained an increasing attention especially in chemical engeenering and genetics (we refer for instance to \cite{BOU07} and to  \cite{MR2457048}). The idea behind PLS is to first reduce the data to a well adapted low dimensional space (which takes into account the covariates and the response at the same time) to then perform estimation and prediction. Originally, it is a sequential procedure that leads to orthogonal latent components maximizing both the variance of the predictors and the covariance with the response variable. The number of components plays the role of the regularizer parameter and is usually chosen by cross validation. The PLS estimator is then defined by applying ordinary least squares to the latent components. Early references on PLS are \cite{NAES85}, \cite{HE88}, \cite{HE90}, \cite{MA92} and \cite{FRANCK93}. For more details on PLS we also refer to \cite{HE01} and \cite{ROS06}.

  PLS has been mainly investigated but its statistical properties are still little known. This is mainly due to the fact that this estimator depends in a non linear way of the response. We have developped a new approach for the study of the PLS properties (cf. \cite{BLAZERE}) based on the connections between PLS and orthogonal polynomials. In this paper we consider again these connections to provide a general and unified framework for the study of the PLS properties. Using this approach, we show that we can easily recover proofs of results on PLS proved earlier by several authors through various approaches. In this paper, we will also explain how our approach sheds new lights on the method and is powerful to gain more insight into the PLS properties.

Now let us detail the ouline of this paper. First, in Section \ref{section:framework}, we set the framework and the notations. Then, we recall in Section \ref{section:PLS} the main idea behind PLS and one of the main property of the associated estimator, that is its link with the Krylov subspaces. We also recall the connections between PLS and orthogonal polynomials. We introduce the residual polynomials, discuss their main properties and give their analytical expression stated in a previous paper (see \cite{BLAZERE}).
In Section \ref{section:filter}, we derive a new formula for the PLS filter factors which only depends on the residual polynomials. Using this new expression, we show how it is obvious to recover most of the main properties of these PLS filter factors (\cite{LIN00}, \cite{BUTLER00}). Section \ref{section:global shrinkage} provides a new expression for the PLS estimator in terms again of the residuals polynomials. We also state in this section a slightly modified proof of the fact that PLS is a global shrinkage estimator (\cite{JONG95}, \cite{GOU96}). 
Section \ref{section:empirical risk} investigates the behaviour of the empirical risk. We derive in the PLS frame, for the first time up to our knowledge, an exact analytical expression for the empirical risk in terms of signal and noise. From this expression we provide a more in depth analysis of the empirical risk. In particular we show that it is then straightforward to prove that PLS provides a better fit than Principal Components Regression (\cite{JONG93}, \cite{PHATAK02}).
Finally, in Section \ref{section:MSE}, we study the Mean Squares Error (MSE) of the PLS estimator. The decomposition of the MSE, stated in this section, highlights the similarities but also the differences between PLS and estimators with deterministic filter factors.

\section{Framework}
\label{section:framework}
\subsection{The regression model}
We denote by $A^T$ the transpose matrix of $A$ and by $I$ the identity matrix (we forget the index when there is no confusion concerning the size of the matrix). The Euclidean norm is denoted by $\| . \|$.

We consider the classical linear regression model 
\begin{equation}
\label{eq:regression-model}
Y=X\beta^{*}+\varepsilon
\end{equation}
 where $X$ is a $(n,p)$ matrix which contains the data and whose columns are the covariates. Each row represents an observation. The design matrix can be fixed or random.  $Y=(Y_{1},...,Y_{n})^{T} \in \mathbb{R}^{n}$ is the vector of the observed outcome also called the response. $\beta^{*}=(\beta_{1}^*,...,\beta_{p}^*)^{T} \in \mathbb{R}^{p} $ is the unknown parameter vector. The vector $\varepsilon=(\varepsilon_{1},...,\varepsilon_{n})^{T} \in \mathbb{R}^{n}$ contains the errors. The errors are assumed to be independent, centered, identically distributed with common variance $\sigma^{2}$ (we do not assume that the errors are Gaussian). To simplify we assume that $X$ and $Y$ are centered in such a way that there is no intercept.
 We allow $p$ to be much larger than $n$ and we denote by $r$ the rank of $X$. We assume that $r=\textrm{min}(n,p)$.

\subsection{Singular value decomposition of the design matrix}
An important and useful tool to study the properties of the PLS estimator is the Singular Value Decomposition (SVD).
The SVD of $X$ is given by $$X=UDV^{T}$$ where
\begin{itemize}
\item $U$ is a $(n,n)$ matrix and $U^{T}U=UU^{T}=I$. This means that the columns $u_{1},...,u_{n}$ of $U$ form an orthonormal basis of $\mathbb{R}^n$.
\item $V$ is a $(p,p)$ matrix and $V^{T}V=VV^{T}I$. So the columns $v_{1},...,v_{p}$ of $V$ form an orthonormal basis of $\mathbb{R}^p$.
\item $D\in \mathbb{M}_{n,p}$ is a matrix which contains $(\sqrt{\lambda_{1}},...,\sqrt{\lambda_{r}})$ on the diagonal and zero anywhere else (i.e. $d_{ii}=\sqrt{\lambda_i}$ for $i=1,...,r$ and $d_{ij}=0$ otherwise).
\end{itemize}
$\lambda_1,...,\lambda_r$ represent the non-zero positive eigenvalues of the predictor sample covariance matrix $X^TX$. Without loss of generality we assume that $\lambda_{1}\geq \lambda_{2}\geq ....\geq\lambda_{r}>0$. Of course when the design matrix is random the eigenelements of $X$ i.e. $\left( \lambda_i, u_i, v_i\right)$ are random too.

 \begin{notation}
 We denote by $\tilde{\varepsilon}_{i}:= \varepsilon^{T}u_{i}$, $i=1,...,n$ and $\tilde{\beta}^{*}_{i}:={\beta^{*}}^{T}v_{i}$, $i=1,...,p$ the projections of $\varepsilon$ and $\beta^{*}$ respectively onto the left and right eigenvectors of $X$. We also define two important quantities that appear frequently in the study of the PLS properties:
\begin{itemize}
 \item  $p_i=(X\beta^*)^Tu_i,\: i=1,...,n$
 \item   $\hat{p}_i=Y^Tu_i,\: i=1,...,n$.
 \end{itemize}

 \end{notation}

\section{Connections between PLS and discrete orthogonal polynomials}
\label{section:PLS}
\subsection{Why the PLS method?}

When the covariance matrix $X^TX$ is invertible ($p\leq n$) the Ordinary Least Squares (OLS) estimator of $\beta^*$ is
$$\hat{\beta}_{OLS}=(X^TX)^{-1}X^TY.$$
In many situations (genetics, chemometrics...) $p>n$ or $X^TX$ is ill conditionned because of multicollinearity and therefore $\hat{\beta}_{OLS}$ is not defined.
In this case we can still consider a similar estimator which is the minimum length least squares estimator defined by  
$$\hat{\beta}_{MLLS}:=(X^TX)^{-}X^TY,$$
 where $(X^TX)^{-}$ is the Moore Penrose inverse of $X^TX$ (see \cite{ENGL96}). The Moore Penrose inverse of $X^TX$ is defined by $$(X^TX)^{-}=\sum_{i=1}^{r}\lambda_i^{-1}v_iv_i^T. $$
Of course when $X^TX$ is invertible, $r=p$ and $(X^TX)^{-}=\sum_{i=1}^{p}\lambda_i^{-1}v_iv_i^T=(X^TX)^{-1}$. Hence, we recover the OLS estimator.
Expanding $\hat{\beta}_{MLLS}$ in the right eigenvectors directions gives
$\hat{\beta}_{MLLS}:=\sum_{i=1}^{r}\dfrac{\hat{p}_i}{\sqrt{\lambda_i}}v_i$.
For simplicity we just keep one notation and thus defined the Least Squares estimator as 
$$\hat{\beta}_{LS}=\sum_{i=1}^{r}\dfrac{\hat{p}_i}{\sqrt{\lambda_i}}v_i,$$
where we recall that $\hat{p}_i=Y^Tu_i$.

When some $\lambda_i$ are small the LS estimator as a high variance. In this case it is better to use alternative estimators, based on dimension reduction, like the one given by Principal Components Regression (PCR) (\cite{JO05}). In this case the data are projected only onto the eigenvectors associated with high eigenvalues
$$\hat{\beta}^{m}_{PCR}=\sum_{i=1}^{m}\dfrac{\hat{p}_i}{\sqrt{\lambda_i}}v_i,$$
where $m\leq r$ is the regularizing parameter.

However, in a regression context, PCR can fails in some situations. Indeed, \cite{JO82} highlighted some real-life examples where the principal components corresponding to small eigenvalues have high correlations with $Y$. To avoid this situation one can think of the PLS method. As mentionned before, this procedure takes into account the value of the response to build a low dimensional space by maximazing both the variance of the predictors and the covariance with the response variable. Then, the data are projected into this lower space to sequentially build latent components. For the algorithmic construction we refer to \cite{WO85} and to \cite{FRANCK93}. In our work we do not consider the sequential construction of the PLS components. We rather use that PLS is the minimization of least squares over some Krylov subspaces. 

\begin{prop}\cite{HE88}
\label{prop:CLS}
\begin{equation}
\hat{\beta}_{k}^{PLS}= \underset{\beta \in \mathcal{K}^{k}(X^{T}X, X^{T}Y)}{\textrm{argmin}}\Vert Y-X\beta\Vert^{2}
\label{eq:esti_betabis} 
\end{equation}
where $\mathcal{K}^{k}(X^{T}X,X^{T}Y)=\left\lbrace X^{T}Y, (X^{T}X)X^{T}Y,..., (X^{T}X)^{k-1}X^{T}Y\right\rbrace$.
\end{prop}
The space $\mathcal{K}^{k}(X^{T}X,X^{T}Y)$ spanned by $X^{T}Y, (X^{T}X)X^{T}Y,..., (X^{T}X)^{k-1}X^{T}Y$ (and denoted by $\mathcal{K}^{k}$ when there is no possible confusion) is called the $k^{th}$ Krylov subspace with respect to $X^{T}Y$ and $X^{T}X$ (\cite{SAAD92}). Notice that it is a random subspace because it depends on $Y$.

Proposition \ref{prop:CLS} above shows that the PLS estimator at step $k$ minimizes the least squares over some specific Krylov subspaces of dimension $k$. Notice that, contrary to classical projection methods, the PLS subspaces are random. This makes the PLS approach more difficult to study than the PCR one. Notice further that if $k=r$ then 
$\hat{\beta}_{k}^{PLS}=\hat{\beta}_{LS}$. 
Of course if $X^TX$ is invertible we recover $\hat{\beta}_{OLS}$ for $k=p$.

The maximal dimension of the Krylov subspaces sequence with respect to $k$ is equal to the number of different eigenvalues for which the associated $\hat{p}_i$ are non zero. Of course, this maximal dimension is always lower than the rank $r$ of $X$. Here, we assume that this maximal number is exactly equal to $r$.

\subsection{The discrete orthogonal polynomials approach}
\subsubsection{Link between PLS and discrete orthogonal polynomials  }
In this subsection we denote by $\mathcal{P}_{k}$ the set of all polynomials of degree at most $k$ and by $\mathcal{P}_{k,1}$ the subset of $\mathcal{P}_{k}$ constituted by polynomials with constant term equals to one. To simplify the notations we just denote by $\hat{\beta}_{k}$ the PLS estimator at step $k$.

Proposition \ref{prop:CLS} above is the starting point of our work. In fact, if we consider PLS with this algrebaic point of view, it is easy to see that the PLS estimator has a polynomial representation in terms of $X^TX$. It is a straightforward consequence of the fact that $\hat{\beta}_k \in \mathcal{K}^k$.
\begin{prop}
\label{prop:poly}
For $k\leq r$ we have
\begin{equation}
\label{eq:PLS_poly}
\hat{\beta}_{k}=\hat{P}_{k}(X^{T}X)X^{T}Y
\end{equation}
where $\hat{P}_{k} \in  \mathcal{P}_{k-1}$ and satisfies $$\Vert Y-X\hat{P}_{k}(X^{T}X)X^{T}Y\Vert^{2}=\underset{P \in \mathcal{P}_{k-1}}{\textrm{argmin}}\Vert Y-XP(X^{T}X)X^{T}Y\Vert^{2}$$
and 
\begin{equation}
\label{eq:residu-poly}\Vert Y-X\hat{\beta}_{k}\Vert^{2}=\Vert \hat{Q}_{k}(XX^{T})Y\Vert^{2}
\end{equation}
where $\hat{Q}_{k}(t)=1-t\hat{P}_{k}(t)$ lies in $\mathcal{P}_{k,1}$ and satisfies $\Vert \hat{Q}_{k}(XX^{T})Y\Vert^{2}=\underset{Q\in \mathcal{P}_{k,1}}{\textrm{min}}\Vert  Q(XX^{T})Y\Vert^{2}$.
\end{prop}
The polynomials $\hat{Q}_k$ are called the residual polynomials.

Notice that when $k=r$, $\hat{Q}_r(x)=\prod_{i=1}^{r}(1-\frac{x}{\lambda_i})$. Therefore $\Vert Y-X\hat{\beta}_{r}\Vert^{2}=\sum_{i=r+1}^{n}\hat{p}_i^2$ if $r<n$ and equals zero if $r=n$.\\

Now let us recall the link between the residual polynomials and discrete orthogonal polynomials (see \cite{NIK91} for futher details on discrete orthogonal polynomials). 

Let $\hat{Q}_{0}:=1$
\begin{prop}
$\hat{Q}_{0},\hat{Q}_{1},...,\hat{Q}_{r}$ is a sequence of discrete orthogonal polynomials with respect to the measure 
$$d\hat{\mu}(\lambda)=\sum_{j=1}^{r}\lambda_{j}(u_{j}^{T}Y)^{2}\delta_{\lambda_{j}}.$$
\label{prop: poly-ortho}
\end{prop}
\begin{proof}
For the proof of this proposition we refer to \cite{BLAZERE}.
\end{proof}

The support of the measure $\hat{\mu}$ is the non-zero spectrum of the covariance matrix $X^TX$ and the associated weights are the $\lambda_{j}(u_{j}^{T}Y)^{2}=\left( \sqrt{\lambda_i}\hat{p}_i\right) ^{2}=\left( (Xv_i)^TY\right) ^{2}$. The weight are positive and the magnitude of the point masses correspond in fact to the covariance between the principal components and the response $Y$. Thus, the measure $\hat{\mu}$ captures both the variation in $X$ and the correlation between $X$ and $Y$ along each eigenvector direction.

\subsubsection{Interest of the residual polynomials}
\label{subsection:interest}
Using Proposition \ref{prop:poly} and expanding $X^TX$ and $XX^T$ in terms of the right and left eigenvectors of $X$, we can write most PLS objects just in terms of the eigenelements of $X$ and of the residual polynomials:
\begin{itemize}
 \item $\hat{\beta}_k=\hat{P}_{k}(X^{T}X)X^{T}Y=\sum_{i=1}^{r}\left(1-\hat{Q}_k(\lambda_i)\right) \dfrac{\hat{p}_i}{\sqrt{\lambda_i}}v_i .$
 \item $X\hat{\beta}_k=(I-\hat{Q}_k(XX^T))Y=\sum_{i=1}^{r}\left(1-\hat{Q}_k(\lambda_i)\right) \hat{p}_iu_i .$
 \item $Y-X\hat{\beta}_k=\hat{Q}_k(XX^T)Y=\left\lbrace 
\begin{array}{ccc}
 \sum_{i=1}^{r}\hat{Q}_k(\lambda_i) \hat{p}_iu_i+\sum_{i=r+1}^{n}\hat{p}_iv_i& \mbox{if} &  r<n \\
 \sum_{i=1}^{r}\hat{Q}_k(\lambda_i) \hat{p}_iu_i. & \mbox{if} & r=n
\end{array}\right.
 $
 \end{itemize} 
 because $\hat{Q}_k(0)=1$.
The aim of the next subsection is to provide an expression for the residual polynomials easier to interpret and well tailored to the study of the PLS properties.

\subsubsection{Expression of the residual polynomials}

Based on the theory of orthogonal polynomials and Proposition \ref{prop: poly-ortho}, we can provide an explicit formula for  the residual polynomials $\left( \hat{Q}_{k}\right) _{1\leq k \leq r}$.  This formula clearly shows how the disturbance on the observations and the distribrution of the spectrum impact on the residuals. This expression of the residual polynomials contains all the information necessary to study the PLS properties.

\begin{theorem}
\label{theo: expression-det-2}
Let $k\leq r$ and $$I_{k}^{+}=\left\lbrace ( j_{1},...,j_{k}): r\geq j_{1}>...>j_{k}\geq 1\right\rbrace.$$

We have
\begin{equation}
\label{eq:final-expression-residuals}
\hat{Q}_{k}(x)=\sum_{(j_{1},..,j_{k})\in I^{+}_{k}}\left[ 
\dfrac{\hat{p}_{j_{1}}^{2}...\hat{p}_{j_{k}}^{2}\lambda_{j_{1}}^{2}...\lambda_{j_{k}}^{2}V(\lambda_{j_{1}},...,\lambda_{j_{k}})^{2}}{\sum_{(j_{1},..,j_{k})\in I^{+}_{k}} \hat{p}_{j_{1}}^{2}...\hat{p}_{j_{k}}^{2}\lambda_{j_{1}}^{2}...\lambda_{j_{k}}^{2}V(\lambda_{j_{1}},...,\lambda_{j_{k}})^{2}}\right] \prod_{l=1}^{k}(1-\frac{x}{\lambda_{j_{l}}}).
\end{equation}
where  we recall that $\hat{p}_{i}:= p_{i}+ \tilde{\varepsilon}_{i} $ with $p_{i}:=(X\beta^{*})^{T}u_{i}=\sqrt{\lambda_{i}}\tilde{\beta}^{*}_{i}$ and $\tilde{\varepsilon}_{i}:= \varepsilon^{T}u_{i}$.
\end{theorem}
\begin{proof}
We refer again to \cite{BLAZERE} for the proof of this theorem.
\end{proof}
The right hand side of Equation (\ref{eq:final-expression-residuals}) is of course a polynomial of degree $k$ and is equal to one at zero.

A look to the expression of the residual polynomials in Theorem \ref{theo: expression-det-2} gives a better understanding of the PLS complexity. In fact, contrary to PCR, all the eigenvectors directions are taken into account at each step. Besides the residuals depend in a complicated way on the response through the normalization and the product of the $\hat{p}_i$. However, we can give an interpretation of this formula easier to understand.

Indeed, for all $(j_{1},...,j_{k}) \in I_{k}^{+}$, let $$\hat{w}_{j_{1},..,j_{k}}:=\dfrac{\hat{p}_{j_{1}}^{2}...\hat{p}_{j_{k}}^{2}\lambda_{j_{1}}^{2}...\lambda_{j_{k}}^{2}V(\lambda_{j_{1}},...,\lambda_{j_{k}})^{2}}{\sum_{(j_{1},..,j_{k})\in I^{+}_{k}} \hat{p}_{j_{1}}^{2}...\hat{p}_{j_{k}}^{2}\lambda_{j_{1}}^{2}...\lambda_{j_{k}}^{2}V(\lambda_{j_{1}},...,\lambda_{j_{k}})^{2}}.$$
We define $Z_k:=\sum_{(j_{1},..,j_{k})\in I^{+}_{k}} \hat{p}_{j_{1}}^{2}...\hat{p}_{j_{k}}^{2}\lambda_{j_{1}}^{2}...\lambda_{j_{k}}^{2}V(\lambda_{j_{1}},...,\lambda_{j_{k}})^{2}$ as the normalized constant. 

Then, we have $$\hat{Q}_{k}(x)=\sum_{(j_{1},..,j_{k})\in I^{+}_{k}} \left[ \hat{w}_{(j_{1},..,j_{k})}\prod_{l=1}^{k}(1-\frac{x}{\lambda_{j_{l}}})\right] ,$$
where $\prod_{l=1}^{k}(1-\frac{x}{\lambda_{j_{l}}})$ lies again in $\mathcal{P}_{k,1}$. Furthermore, its roots $\lambda_{j_{1}},...,\lambda_{j_{k}}$ are members of the spectrum of $XX^{T}$.
Notice that $$0<\hat{w}_{(j_{1},..,j_{k})}\leq 1$$ and $$\sum_{(j_{1},..,j_{k})\in I^{+}_{k}}\hat{w}_{(j_{1},..,j_{k})}=1.$$
So we can interpret the weights $(\hat{w}_{(j_{1},..,j_{k})})_{I^{+}_{k}}$ as probabilities on $\mathcal{P}_{k,1}$ supported by polynomials having their roots in the spectrum of the design matrix.

Therefore $\hat{Q}_{k}(\lambda_{i})$ is the sum over all elements in $I_{k}^{+}$ of $\prod_{l=1}^{k}(1-\frac{x}{\lambda_{j_{l}}})$ weighted by the probabilities $\hat{w}_{(j_{1},..,j_{k})}$. In other words, it is the convex combinaison of all the polynomials in $\mathcal{P}_{k,1}$ whose roots are subsets of $\left\lbrace \lambda_{1},...,\lambda_{n}\right\rbrace $. The Vandermonde determinant, in the weights, means that the probability of a polynomial with multiple roots is zero. The weights themselves are not easy to interpret. However, they are even greater when the magnitude and the distance between the involved eigenvalues are large and the contribution of the response along the associated eigenvectors is important. In particular, polynomials whose roots are  associated to large $\left( \lambda_i^2p_i^2\right) $ have more heavy weight.

It is not possible to compute exactly the expectation of the weights $\hat{w}_{(j_{1},..,j_{k})}$ (because of the normalized constant) but we can provide a first order approximation of $\mathbb{E}\left[\hat{w}_{(j_{1},..,j_{k})}  \right]$ in case of a fixed design matrix.
\begin{lem}
Let $X$ be a fixed design matrix. Then
\label{lem:expectation}
$$0\leq \mathbb{E}\left[\hat{w}_{(j_{1},..,j_{k})}  \right]\simeq \dfrac{\left( p_{j_{1}}^{2}+\sigma^{2}\right) ...\left( p_{j_{k}}^{2}+\sigma^{2}\right) \lambda_{j_{1}}^{2}...\lambda_{j_{k}}^{2}V(\lambda_{j_{1}},...,\lambda_{j_{k}})^{2}}{\sum_{(j_{1},..,j_{k})\in I^{+}_{k}} \left[ \left( p_{j_{1}}^{2}+\sigma^{2}\right) ...\left( p_{j_{k}}^{2}+\sigma^{2}\right)\lambda_{j_{1}}^{2}...\lambda_{j_{k}}^{2}V(\lambda_{j_{1}},...,\lambda_{j_{k}})^{2}\right]}\quad (\textrm{first order} ).$$
\end{lem}
\begin{proof}
Let $S$ and $T$ two random variables such that $T$ either has no mass at $0$ (if discrete) or has support in $\left[ 0,  +\infty \right[  $ (if continuous). The first order Taylor expansion of $f:(s,t)\mapsto \frac{s}{t}$ around $(\mathbb{E}(S), \mathbb{E}(T))$ provides a first order approximation of the expectation of 
$\dfrac{S}{T}$: $$\mathbb{E}\left[\dfrac{S}{T}\right]\simeq \dfrac{\mathbb{E}(S)}{\mathbb{E}(T)} \quad \textrm{(first order)}.$$ 
Applying this result with $$R:=\hat{p}_{j_{1}}^{2}...\hat{p}_{j_{k}}^{2}\lambda_{j_{1}}^{2}...\lambda_{j_{k}}^{2}V(\lambda_{j_{1}},...,\lambda_{j_{k}})^{2}$$ and $$T:=\sum_{(j_{1},..,j_{k})\in I^{+}_{k}} \hat{p}_{j_{1}}^{2}...\hat{p}_{j_{k}}^{2}\lambda_{j_{1}}^{2}...\lambda_{j_{k}}^{2}V(\lambda_{j_{1}},...,\lambda_{j_{k}})^{2} $$ leads to 
\begin{equation}
\label{eq:E1}
\mathbb{E}\left[\hat{w}_{(j_{1},..,j_{k})}  \right]\simeq \dfrac{\mathbb{E}\left[  \hat{p}_{j_{1}}^{2}...\hat{p}_{j_{k}}^{2}\lambda_{j_{1}}^{2}...\lambda_{j_{k}}^{2}V(\lambda_{j_{1}},...,\lambda_{j_{k}})^{2}\right]}{\mathbb{E}\left[\sum_{(j_{1},..,j_{k})\in I^{+}_{k}} \hat{p}_{j_{1}}^{2}...\hat{p}_{j_{k}}^{2}\lambda_{j_{1}}^{2}...\lambda_{j_{k}}^{2}V(\lambda_{j_{1}},...,\lambda_{j_{k}})^{2} \right]}\quad (\textrm{first order} ).
\end{equation}

Let $(j_1,...,j_k)\in I_{k}^{+}$. Because $j_1 \neq ...\neq j_k$, the random variables $\hat{p}_{j_{1}},...,\hat{p}_{j_{k}}$ are independant. Therefore we have
$$ 
 \mathbb{E}\left[  \hat{p}_{j_{1}}^{2}...\hat{p}_{j_{k}}^{2}\lambda_{j_{1}}^{2}...\lambda_{j_{k}}^{2}V(\lambda_{j_{1}},...,\lambda_{j_{k}})^{2}\right]=\mathbb{E}\left(  \hat{p}_{j_{1}}^{2}\right) ...\mathbb{E}\left(\hat{p}_{j_{k}}^{2}\right) \lambda_{j_{1}}^{2}...\lambda_{j_{k}}^{2}V(\lambda_{j_{1}},...,\lambda_{j_{k}})^{2}$$
\begin{equation}
\label{eq:E2}
=\left(  p_{j_{1}}^{2}+\sigma^{2}\right) ...\left(p_{j_{k}}^{2}+\sigma^{2}\right) \lambda_{j_{1}}^{2}...\lambda_{j_{k}}^{2}V(\lambda_{j_{1}},...,\lambda_{j_{k}})^{2}.\end{equation}
Then, by linearity of the expectation, we get
 \begin{equation}
 \label{eq:E3}
 \mathbb{E}\left[\sum_{(j_{1},..,j_{k})\in I^{+}_{k}} \hat{p}_{j_{1}}^{2}...\hat{p}_{j_{k}}^{2}\lambda_{j_{1}}^{2}...\lambda_{j_{k}}^{2}V(\lambda_{j_{1}},...,\lambda_{j_{k}})^{2} \right]=\sum_{(j_{1},..,j_{k})\in I^{+}_{k}} \left[\prod_{l=1}^{k} \left[ \left( p_{j_{l}}^{2}+\sigma^{2}\right) \lambda_{j_{l}}^{2}\right] V(\lambda_{j_{1}},...,\lambda_{j_{k}})^{2}\right]. 
 \end{equation}
 To conclude, from Equation (\ref{eq:E1}), (\ref{eq:E2}) and (\ref{eq:E3}), we deduce Lemma \ref{lem:expectation}.
\end{proof}

\subsubsection{Other properties of the residuals polynomials}
\label{subsection:others_properties}
In this subsection we present other useful properties of the residual polynomials that will be very useful later on in this paper (in particular the second point).
\begin{lem}\
\label{lem:properties}
\begin{enumerate}
\item $\hat{Q}_k$ has $k$ real zeros.
\item $\sum_{i=1}^{r}\hat{Q}_k(\lambda_i)^2\hat{p}_i^2=\sum_{i=1}^{r}\hat{Q}_k(\lambda_i)\hat{p}_i^2.$
\item $\sum_{i=1}^{r}\hat{Q}_j(\lambda_i)\hat{Q}_k(\lambda_i)\hat{p}_i^2=\sum_{i=1}^{r}\hat{Q}_j(\lambda_i)\hat{p}_i^2, \: j\leq k$
\item  $\mid\hat{Q}_{k}(\lambda_{i})\mid\leq \underset{I_{k}^{+}}{\textrm{max}}\left( \prod_{l=1}^{k}\left\lvert 1-\frac{\lambda_{i}}{\lambda_{j_{l}}}\right \rvert\right) $ for all $k\leq r$ and all $1\leq i\leq r$.

In particular if $ \lambda_{1}(1-\varepsilon) \leq\lambda_{i} \leq \lambda_{n}(1+\varepsilon)$ then $\mid\hat{Q}_{k}(\lambda_{i})\mid\leq \varepsilon^{k}.$
\item PLS requires fewer coefficients than PCR. In fact if the $r$ non zero eigenvalues take only $K$ different values then $\hat{Q}_{K+1}:=0$.
\end{enumerate}
\end{lem}
\begin{proof}
\begin{enumerate}
\item Straightforward consequence of the fact that $\left( \hat{Q}_k\right)_{0\leq k \leq r}$ are discrete orthogonal polynomials.
\item $\hat{Q}_k(XX^T)X^TY=\left[ I-\hat{\Pi}_{k}\right] Y$ where $\hat{\Pi}_k$ is the orthogonal projector onto the space spanned by $\mathcal{K}^{k}(XX^T,XX^TY)$. So $$\sum_{i=1}^{r}\hat{Q}_k(\lambda_i)^2\hat{p}_i^2=\parallel Y-\hat{\Pi}_{k}Y\parallel^{2}=Y^T\left( I-\hat{\Pi}_{k}\right) Y=\sum_{i=1}^{r}\hat{Q}_k(\lambda_i)\hat{p}_i^2.$$
\item Similar argument based on $\hat{\Pi}_{k}\hat{\Pi}_{j}=\hat{\Pi}_{j}$ because $\mathcal{K}^{j}(XX^T,XX^TY)\subset \mathcal{K}^{k}(XX^T,XX^TY)$.
\item See formula (\ref{eq:final-expression-residuals}).
\item See formula (\ref{eq:final-expression-residuals}).
\end{enumerate}
\end{proof}

\section{Filter factors}
In this section we investigate the shrinkage properties of the PLS estimator.
\label{section:filter}
\subsection{ New expression for the filter factors}
We recall that 
\begin{equation}
\label{eq:expand_PLS}
\hat{\beta}_{k}=\sum_{i=1}^{r}(1-\hat{Q}_k(\lambda_i))\dfrac{\hat{p}_i}{\sqrt{\lambda_i}}v_i.
\end{equation}
From this decomposition of $\hat{\beta}_{k}$, we deduce that the filter factors $f_i^{(k)}$ of the PLS estimator relative to OLS are equals to $f_i^{(k)}:=1-\hat{Q}_k(\lambda_i)$.
We recall that the filter factors are the weights associated to the expansion of $\hat{\beta}_{LS}$ with respect to the eigenvectors directions of the covariance matrix (see \cite{LIN00} for further details on the filter factors). Therefore, we have an alternative representation of the filter factors in terms of the residual polynomials. Indeed, using Theorem \ref{theo: expression-det-2}, we can expand the filter factors and provide a new expression as follow:
\begin{equation}
\label{eq:new_filter_factor}
f_i^{(k)}:=\sum_{(j_{1},..,j_{k})\in I^{+}_{k}} \hat{w}_{(j_{1},..,j_{k})}\left[ 1-\prod_{l=1}^{k}(1-\frac{\lambda_i}{\lambda_{j_{l}}})\right], 
\end{equation}
where we recall that $\hat{w}_{j_{1},..,j_{k}}:=\dfrac{\hat{p}_{j_{1}}^{2}...\hat{p}_{j_{k}}^{2}\lambda_{j_{1}}^{2}...\lambda_{j_{k}}^{2}V(\lambda_{j_{1}},...,\lambda_{j_{k}})^{2}}{\sum_{(j_{1},..,j_{k})\in I^{+}_{k}} \hat{p}_{j_{1}}^{2}...\hat{p}_{j_{k}}^{2}\lambda_{j_{1}}^{2}...\lambda_{j_{k}}^{2}V(\lambda_{j_{1}},...,\lambda_{j_{k}})^{2}}.$
 
This is an alternative representation to the one of \cite{LIN00}
who consider the following implicit expression for the filter factors (see Theorem 1 in \cite{LIN00}) to study the shrinkage properties of PLS:
\begin{equation}
\label{eq:new_ff}
f_i^{(k)}=\dfrac{(\theta_{1}^{(k)}-\lambda_i)...(\theta_{k}^{(k)}-\lambda_i)}{\theta_{1}^{(k)}...\theta_{k}^{(k)}}
\end{equation}
where $(\theta_{i}^{(k)})_{1\leq i\leq k}$ are the eigenvalues of $W_{k}({W_{k}}^{T}\Sigma W_{k}){W_{k}}^{T}$ called the Ritz eigenvalues. 
The interest of Formula (\ref{eq:new_filter_factor}) compared to (\ref{eq:new_ff}) is that it clearly and explicitely shows how the filter factors depend on the error terms and on the eigenelements of $X$. We can notice that they are completely determined by these last quantities.

From Equation (\ref{eq:new_filter_factor}), we easily see that the PLS filter factors are polynomials of degree $k$ that strongly depend on the response in a non linear and complicated way (product of the projections of the response onto the right eigenvectors and normalization factor). Furthermore, because the PLS filter factors are stochastics, usual results for linear spectral methods such as PCA or Ridge, cannot be applied in this case. Contrary to those of PCA or Ridge regression, the PLS filter factors are not easy to interpret. This is closely linked to the intrinsic idea of the method that takes into account at the same time the variance of the explanatory variables and their covariance with the response. However, we have a control of the distance of the filter factors to one.
\begin{prop}
For all $k\leq r$, we have
$$\left | 1-f_i^{(k)}(\lambda_i)\right| \leq \left( \dfrac{\lambda_1-\lambda_r}{\lambda_r}\right) ^{n}\left( 1+\hat{p}_{i}^{2}\lambda_{i}^{2}\frac{\sum_{I_{k-1,i}^{+}} \hat{p}_{j_{1}}^{2}...\hat{p}_{j_{k-1}}^{2}\lambda_{j_{1}}^{2}...\lambda_{j_{k-1}}^{2}V(\lambda_{j_{1}},...,\lambda_{j_{k-1}})^{2}}{\sum_{I_{k,i}^{+}} \hat{p}_{j_{1}}^{2}...\hat{p}_{j_{k}}^{2}\lambda_{j_{1}}^{2}...\lambda_{j_{k}}^{2}V(\lambda_{j_{1}},...,\lambda_{j_{k}})^{2}}\right) ^{-1},$$
where $I_{k,i}^{+}:=\left\lbrace (j_1,...,j_k)\in I_k^{+} \mid j_l \neq i, l=1,...,k\right\rbrace$.
\end{prop}
 So the highest are the $\lambda_i$ and $\hat{p}_i$ the closest to one is $f_i^{(k)}$ and the largest is the amount of expansion in this eigenvector direction. 
Actually, the PLS filter factors  are not only related to the singular values but also to the magnitude of the covariance between the principal components and the response: what seems to be important it is not the order of decrease fo $\lambda_i$ but the order of decrease of $\lambda_i \hat{p}_i^{2}$.

 \subsection{Shrinkage properties of PLS: other proofs of known results}
 In this subsection, we explain how we can easily recover (once Theorem \ref{theo: expression-det-2} is stated) most of the main known results on the PLS filter factors.
 \begin{enumerate}
 \item 
From Formula (\ref{eq:new_filter_factor}), we easily see that there is no order on the filter factors and no link between them at each step. Furthermore, they are not always in $\left[ 0,1\right]$, contrary to those of PCR or Ridge regression. These last methods always shrink in all the eigenvectors directions. In particular the PLS filter factors can be greater than one and even negative. This is one of their very particular feature. PLS shrinks in some direction but can also expand in others in such a way that $f_{i}^{(k)}$ represents the magnitude of shrinkage or expansion of the PLS estimator in the  $i^{th}$ eigenvectors direction. \cite{FRANCK93} were the first to notice this peculiar property of PLS but they did not provide any proof. This result was first proved by \cite{BUTLER00} and independantly the same year by \cite{LIN00} using Ritz eigenvalues. We also refer to \cite{KRA07} for an overview of the shrinkage properties of the PLS estimator.

The shrinkage properties of the PLS estimator were mainly investigated by \cite{LIN00}. From Formula (\ref{eq:new_filter_factor}), we easily recover their main properties for the filter factors (but without using the Ritz eigenvalues). It is for instance the case for the behaviour of the filter factors associated to the largest and smallest eigenvalue. Indeed, on one hand, if $k\leq r$ and $i=r$ then $0<\prod_{l=1}^{k}(1-\frac{\lambda_{r}}{\lambda_{j_{l}}})<1$. Therefore, because $\sum_{(j_{1},..,j_{k})\in I^{+}_{k}}\hat{w}_{(j_{1},..,j_{k})}=1$, we can conclude directly that $0<f_r^{(k)}<1$.\\
On the other hand, if $k\leq r$ and $i=1$ then 	
$$
 \left\lbrace 
\begin{array}{ccc}
\prod_{l=1}^{k}(1-\frac{\lambda_{1}}{\lambda_{j_{l}}})<0 & \mbox{if} & k \: \mbox{is odd} \\
\prod_{l=1}^{k}(1-\frac{\lambda_{1}}{\lambda_{j_{l}}})>0 & \mbox{if} & k \: \mbox{is even}
\end{array}\right.,
$$
so that
$$
\left\lbrace
\begin{array}{ccc}
f_1^{(k)}>1 & \mbox{if} & k \: \mbox{is odd} \\
f_1^{(k)}<1 & \mbox{if} & k \: \mbox{is even}
\end{array}\right..
$$
This is exactly Theorem 3 of \cite{LIN00}. 

Hence, the filter factor associated to the largest eigenvalues oscillates around one depending on the parity of the index of the factors. 
For the other filter factors we can have either $f_i^{(k)}\leq 1$ (PLS shrinks) or $f_i^{(k)}\geq 1$ (PLS expands) depending on the distribution of the spectrum. Notice that if PLS does not shrink along an eigenvector direction (i.e $\mid f_i^{k}\mid> 1$) but $\sqrt{\lambda_i}$ is high or $\hat{p}_i$ is small in this direction then it has not a lot of effect (cf. Equation (\ref{eq:expand_PLS})). In addition, as noticed by \cite{KRA07}, even if $\mid f_i^{k}\mid>1$ this does not always imply that the MSE is worse compared to the one of OLS because the PLS filter factors are stochatics. We will shed new light on this remark in Section \ref{section:MSE} where we will provide a further study of the MSE.
\item
Notice that for orthogonal polynomials of a finite supported measure there exists a point of the support of the discrete measure between any two of their zeros (\cite{BAIK07}). Moreover, the roots of these polynomials belong to the interval whose bounds are the extreme values of the support of the discrete measure.
Therefore, from Proposition \ref{prop: poly-ortho} we deduce that all the $k$ zeros of $\hat{Q}_k$ lie in $\left[\lambda_r, \lambda_1 \right] $ and no more than one zeros lies in $\left[\lambda_{i}, \lambda_{i-1}\right] $, where $i=1,...,r+1$ and by convention $\lambda_{r+1}:=0$ and $\lambda_{0}:= + \infty$. We immediately deduce that the eigenvalues $\left[\lambda_r, \lambda_1 \right] $ can be partitioned into $k+1$ consecutive  disjoint non empty intervals denoted by $\left( I_l\right)_{1\leq l\leq k+1}$ that first shrink and then alternately expand or shrink the OLS. In other words
$$
\left\lbrace
\begin{array}{ccc}
f_i^{(k)}\leq 1 & \mbox{if} &  \lambda_i \in I_l,\quad l \: \textrm{odd}\\
f_i^{k}\geq 1 & \mbox{if} & \lambda_i \in I_l,\quad l \: \textrm{even}
\end{array}\right..
$$
This is Theorem 1 of \cite{BUTLER00}. Notice that this result has been also independently by \cite{LIN00} using the Ritz eigenvalues theory (see Theorem 4).
\item
Besides, if we have $\lambda_{i}< \lambda_{n}(1+\epsilon)$ then a straightforward calculation, based on Formula (\ref{eq:final-expression-residuals}), leads to 
$f_i^{k}< 1+\epsilon^{k}$. This statement is Theorem 7 of \cite{LIN00}.
\item
Furthermore, we also recover Theorem 2 of \cite{BUTLER00}:
\begin{theorem}
For $i=1,...,n$
$$f_i^{(r-1)}=1-C\left( \hat{p}_i \lambda_i\displaystyle{ \prod_{i \neq j}}(\lambda_j-\lambda_i)\right) ^{-1},$$
where $C$ does not depnd on $i$.
\end{theorem}
In addition we have the exact expression for the constant $C$ which is equal to $\dfrac{\prod_{j=1}^{r}\left( \hat{p}_{j}^{2}\lambda_{j}\right)V(\lambda_{1},...,\lambda_{r})^{2}}{\sum_{i=1}^{r}\left[ \prod_{j=1}^{r}\left( \hat{p}_{j}^{2}\lambda_{j}^{2}\right)V(\lambda_{1},...,\lambda_{r})^{2}\right] }$.
\begin{proof}
Based on Formula (\ref{eq:new_filter_factor}), we have
$$f_i^{(r-1)}=1-\dfrac{\prod_{j=1,j\neq i}^{r}\left( \hat{p}_{j}^{2}\lambda_{j}^{2}\right)V(\lambda_{1},...,\lambda_{i-1},...,\lambda_{i+1},...,\lambda_{r})^{2}\prod_{j=1,j\neq i}^{r}(1-\frac{\lambda_i}{\lambda_{j}})}{\sum_{l=1}^{r}\left[ \prod_{j=1, j \neq l}^{r}\left( \hat{p}_{j}^{2}\lambda_{j}^{2}\right)V(\lambda_{1},...,\lambda_{l-1},...,\lambda_{l+1},...,\lambda_{r})^{2}\right] }$$
$$=1-\dfrac{\prod_{j=1,j\neq i}^{r}\left( \hat{p}_{j}^{2}\lambda_{j}(\lambda_j-\lambda_i)^{-1}\right)V(\lambda_{1},...,\lambda_{r})^{2}}{\sum_{i=1}^{r}\left[ \prod_{j=1}^{r}\left( \hat{p}_{j}^{2}\lambda_{j}^{2}\right)V(\lambda_{1},...,\lambda_{r})^{2}\right] }$$
$$=1-\left( \hat{p}_{i}^{2}\lambda_{i}\prod_{j=1,j\neq i}^{r}(\lambda_j-\lambda_i)\right) ^{-1}\dfrac{\prod_{j=1}^{r}\left( \hat{p}_{j}^{2}\lambda_{j}\right)V(\lambda_{1},...,\lambda_{r})^{2}}{\sum_{i=1}^{r}\left[ \prod_{j=1}^{r}\left( \hat{p}_{j}^{2}\lambda_{j}^{2}\right)V(\lambda_{1},...,\lambda_{r})^{2}\right] }.$$
\end{proof}
So the highest is $\hat{p}_{i}^{2}\lambda_{i}\prod_{j=1,j\neq i}^{r}(\lambda_j-\lambda_i)$ the closest to one is $f_i^{(r-1)}$.
 \end{enumerate}

In conclusion, we have showed that, based on our new expression of the PLS filter factors, we easily recover some of their main properties. Thanks to our approach we provide a unified background to all these results. 

\cite{LIN00} mentionned that, using their approach based on the Ritz eigenvalues, it appears difficult to establish the fact that PLS shrinks in a global sense. \cite{BUTLER00} also considered the shrinkage properties of the PLS estimator along the eigenvector directions but as \cite{LIN00}
they did not prove that the PLS estimator is a global shrinkage estimator. With our approach we are able to prove this fact too. This is the aim of the next section.

\section{Global shrinkage estimator} 
\label{section:global shrinkage}
As seen in the previous section, PLS can expand the LS in some eigendirections leading to an increase of the LS estimator's projected length in these directions. But, globally, it is  considered as a shrinkage estimator (as Ridge or PCA estimators) in the sense that its Euclidean norm is lower than the one of the OLS estimator:
\begin{prop}
For all $k\leq r$, we have
$$\parallel \hat{\beta}_k\parallel^{2}\leq \parallel \hat{\beta}_{OLS}\parallel^{2}.$$
\end{prop}
This global shrinkage feature of PLS was first proved algebraically by \cite{JONG95} and a year later \cite{GOU96} proposed a new independant proof based on the PLS iterative construction algorithm by taking a geometric point of view. In addition \cite{JONG95} proved the more stronger following result:
\begin{lem}
\label{lem:GS}
 $\parallel \hat{\beta}_{k-1}\parallel^{2}\leq \parallel \hat{\beta}_{k}\parallel^{2}$ for all $k\leq r$.
 \end{lem}
Two other proofs of this fact were provided later by \cite{PHATAK02}. The first one uses the link between PLS and Conjugate Gradient while the other uses the theory of quadratic forms. We provide below an alternative proof of Lemma (\ref{lem:GS}) using the residual polynomials. This proof is very closed to the one of \cite{PHATAK02} and we do not have to make use of the expression of the residuals to prove it.
  
\begin{proof}
 The vectors $X^T\hat{Q}_0(XX^T)Y$,...,$X^T\hat{Q}_{k-1}(XX^T)Y$ belongs to $\mathcal{K}^{k}(X^{T}X,X^TY)$ and are orthogonals (because $(\hat{Q}_{k})_{0\leq k\leq r}$ is a sequence of orthogonal polynomials with respect to the discrete measure $\hat{\mu}$). Therefore, they formed an orthogonal basis for  $\mathcal{K}^{k}(X^{T}X,X^TY)$. As $\hat{\beta}_k \in \mathcal{K}^{k}(X^{T}X,X^TY) $, we have
$$\parallel\hat{\beta}_k\parallel^2:= \sum_{j=0}^{k-1}\dfrac{\left(\hat{\beta}_k^TX^T\hat{Q}_{j}(XX^T)Y \right)^2 }{\parallel  X^T\hat{Q}_{j}(XX^T)Y \parallel^2}.$$
Further, because $X\hat{\beta}_k=\sum_{i=1}^{r}(1-\hat{Q}_k(\lambda_i))\hat{p}_iu_i$, we may write
 $$\hat{\beta}_k^TX^T\hat{Q}_{j}(XX^T)Y=\sum_{i=1}^{r}(1-\hat{Q}_k(\lambda_i))\hat{Q}_j(\lambda_i)\hat{p}_i^2 =\sum_{i=1}^{r}\hat{Q}_j(\lambda_i)\hat{p}_i^2-\sum_{i=1}^{r}\hat{Q}_k(\lambda_i)\hat{p}_i^2$$
 $$=\parallel Y-X\hat{\beta}_j\parallel^2-\parallel Y-X\hat{\beta}_k\parallel^2=\parallel X\hat{\beta}_k\parallel^2-\parallel X\hat{\beta}_j\parallel^2.$$
For the justification of the equalities above, we refer to Subsection \ref{subsection:interest} and to the second point of Lemma \ref{lem:properties} of Section \ref{section:PLS}.
To conlude $$\parallel \hat{\beta}_k\parallel^2:= \sum_{j=0}^{k-1}\dfrac{\left(\parallel X\hat{\beta}_k\parallel^2-\parallel X\hat{\beta}_j\parallel^2\right)^2 }{\parallel  X^T\hat{Q}_{j}(XX^T)Y \parallel^2}$$
Furthermore, for $1\leq l<k\leq r$, we have $\parallel X\hat{\beta}_l\parallel^2< \parallel X\hat{\beta}_k\parallel^2$ (because  $X\hat{\beta}_l$ and  $X\hat{\beta}_l$ are the orthogonal projection of $Y$ onto two Krylov subspaces, the first one included in the other). So that, we may deduce that $$\parallel\hat{\beta}_k\parallel^2\leq \sum_{j=0}^{k-1}\dfrac{\left(\parallel X\hat{\beta}_{k+1}\parallel^2-\parallel X\hat{\beta}_j\parallel^2\right)^2 }{\parallel  X^T\hat{Q}_{j}(XX^T)Y \parallel^2}:=  \parallel\hat{\beta}_{k+1}\parallel^2.$$
 Finally, because $\parallel\hat{\beta}_r\parallel^2=\parallel\hat{\beta}_{LS}\parallel^2$, we conclude that for all $k\leq r$ we have $$\parallel\hat{\beta}_{k-1}\parallel^2\leq \parallel\hat{\beta}_k\parallel^2 \leq \parallel\hat{\beta}_{LS}\parallel^2.$$
\end{proof}

\section{Empirical risk}
As far as we know, the PLS empirical risk has not been much studied. In fact the geometric point of view of \cite{GOU96} or the one of \cite{LIN00} based on the Ritz eigenvalues are not well tailored to this study. In this section, we give a nice expression of the empirical risk in terms of the eigenelements of $X$ and on the noise on the observations.
\label{section:empirical risk}
\subsection{An analytical expression for the empirical risk}
The empirical risk is defined as $\parallel Y-X\hat{\beta}_k\parallel^{2}$. It quantifies the fit of the model to the data set used. For PLS, we may write $$\parallel Y-X\hat{\beta}_k\parallel^{2}=\sum_{i=1}^{r}\hat{Q}_k(\lambda_i)\hat{p}_i^2+\left\lbrace 
\begin{array}{ccc}
 0 & \mbox{if} &  r=n \\
 \sum_{i=r+1}^{n}\hat{p}_i^2 & \mbox{if} & r<n
\end{array}\right ..$$ However, this expression is not very enlighting. In this section, we provide an analytical expression for the empirical risk which will be more useful to derive important properties of the empirical risk. In particular we will see that based on this new expression it is easy to show that PLS fits closer than PCR.
\begin{prop}
\label{prop:empirical_risk}
For $k<r$
$$\parallel Y-X\hat{\beta}_k\parallel^{2}=$$
%
\begin{equation}
\label{eq:empirical_risk}
\sum_{r>j_1>...>j_k\geq 1}\left[\dfrac{ \hat{p}_{j_{1}}^{2}...\hat{p}_{j_{k+1}}^{2}\lambda_{j_{1}}^{2},...,\lambda_{j_{k}}^{2}V(\lambda_{j_{1}},...,\lambda_{j_{k}})^{2}}{\sum_{(j_{1},..,j_{k})\in I^{+}_{k}} \hat{p}_{j_{1}}^{2}...\hat{p}_{j_{k}}^{2}\lambda_{j_{1}}^{2}...\lambda_{j_{k}}^{2}V(\lambda_{j_{1}},...,\lambda_{j_{k}})^{2}}\sum_{i=j_1+1}^{r} \left( \prod \left(1-\frac{\lambda_i}{\lambda_{j_l}} \right)^2\hat{p}_i^2\right) \right] 
\end{equation}
$$+
\left\lbrace 
\begin{array}{ccc}
 0 & \mbox{if} &  r=n \\
 \sum_{i=r+1}^{n}\hat{p}_i^2 & \mbox{if} & r<n
\end{array}\right ..$$
\end{prop}

Notice that for $k=r$, $\parallel Y-X\hat{\beta}_r\parallel^{2}=
\left\lbrace 
\begin{array}{ccc}
 0 & \mbox{if} &  r=n \\
 \sum_{i=r+1}^{n}\hat{p}_i^2 & \mbox{if} & r<n
\end{array}\right ..$

\begin{proof}
On one hand (cf. Subsection \ref{subsection:interest}), we have 
$$\parallel Y-X\hat{\beta}_k\parallel^{2}=\sum_{i=1}^{r}\hat{Q}_k(\lambda_i)\hat{p}_i^2+\left\lbrace 
\begin{array}{ccc}
 0 & \mbox{if} &  r=n \\
 \sum_{i=r+1}^{n}\hat{p}_i^2 & \mbox{if} & r<n
\end{array}\right ..$$
And on the other hand, using Formula (\ref{eq:final-expression-residuals}), we have
$$\sum_{i=1}^{r}\hat{Q}_k(\lambda_i)\hat{p}_i^2=$$
$$\sum_{i=1}^{r}\left[ \sum_{(j_{1},..,j_{k})\in I^{+}_{k}}\left[ 
\dfrac{\hat{p}_{j_{1}}^{2}...\hat{p}_{j_{k}}^{2}\lambda_{j_{1}}^{2}...\lambda_{j_{k}}^{2}V(\lambda_{j_{1}},...,\lambda_{j_{k}})^{2}}{\sum_{(j_{1},..,j_{k})\in I^{+}_{k}} \hat{p}_{j_{1}}^{2}...\hat{p}_{j_{k}}^{2}\lambda_{j_{1}}^{2}...\lambda_{j_{k}}^{2}V(\lambda_{j_{1}},...,\lambda_{j_{k}})^{2}}\right] \prod_{l=1}^{k}(1-\frac{\lambda_{i}}{\lambda_{j_{l}}})\right] \hat{p}_i^2$$
$$=\dfrac{\sum_{i=1}^{r}\sum_{(j_{1},..,j_{k})\in I^{+}_{k}}\left[ \hat{p}_{j_{1}}^{2}...\hat{p}_{j_{k}}^{2}\lambda_{j_{1}}^{2}...\lambda_{j_{k}}^{2}V(\lambda_{j_{1}},...,\lambda_{j_{k}})^{2} \prod_{l=1}^{k}(1-\frac{\lambda_{i}}{\lambda_{j_{l}}}) \hat{p}_i^2\right] }{\sum_{(j_{1},..,j_{k})\in I^{+}_{k}} \hat{p}_{j_{1}}^{2}...\hat{p}_{j_{k}}^{2}\lambda_{j_{1}}^{2}...\lambda_{j_{k}}^{2}V(\lambda_{j_{1}},...,\lambda_{j_{k}})^{2}}$$
where 
$$\sum_{i=1}^{r}\sum_{(j_{1},..,j_{k})\in I^{+}_{k}}\left[ \hat{p}_{j_{1}}^{2}...\hat{p}_{j_{k}}^{2}\lambda_{j_{1}}^{2}...\lambda_{j_{k}}^{2}V(\lambda_{j_{1}},...,\lambda_{j_{k}})^{2} \prod_{l=1}^{k}(1-\frac{\lambda_{i}}{\lambda_{j_{l}}}) \hat{p}_i^2\right] =$$
$$\sum_{i=1}^{r}\sum_{(j_{1},..,j_{k})\in I^{+}_{k}}\left[ \hat{p}_{j_{1}}^{2}...\hat{p}_{j_{k}}^{2}\lambda_{j_{1}}^{2}...\lambda_{j_{k}}^{2}V(\lambda_{j_{1}},...,\lambda_{j_{k}})\left( \sum_{\sigma \in S(1,...,k)} \epsilon(\sigma)\lambda_{j_{\sigma(2)}}\lambda_{j_{\sigma(k)}}^{k-1}\right) \prod_{l=1}^{k}(1-\frac{\lambda_{i}}{\lambda_{j_{l}}}) \hat{p}_i^2\right] $$
because $V(\lambda_{j_{1}},...,\lambda_{j_{k}})=\sum_{\sigma \in S(1,...,k)} \epsilon(\sigma)\lambda_{j_{\sigma(2)}}\lambda_{j_{\sigma(k)}}^{k-1}$
$$=\sum_{i=1}^{r}\sum_{j_1=1}^r....\sum_{j_k=1}^r\left[ \hat{p}_{j_{1}}^{2}...\hat{p}_{j_{k}}^{2}\lambda_{j_{1}}^{2}...\lambda_{j_{k}}^{2}V(\lambda_{j_{1}},...,\lambda_{j_{k}})\lambda_{j_{2}}\lambda_{j_{k}}^{k-1}\prod_{l=1}^{k}(1-\frac{\lambda_{i}}{\lambda_{j_{l}}}) \hat{p}_i^2\right] $$
$$=\sum_{i=1}^{r}\sum_{j_1=1}^r....\sum_{j_k=1}^r\left[ \hat{p}_{j_{1}}^{2}...\hat{p}_{j_{k}}^{2}\lambda_{j_{1}}...\lambda_{j_{k}}^{k}V(\lambda_{j_{1}},...,\lambda_{j_{k}})\prod_{l=1}^{k}(\lambda_{j_{l}}-\lambda_{i}) \hat{p}_i^2\right].$$
Then, replacing the indices $(i,1,...,k)$ by $(1,2,..,k+1)$ we obtain
$$\sum_{i=1}^{r}\sum_{j_1=1}^r....\sum_{j_k=1}^r\left[ \hat{p}_{j_{1}}^{2}...\hat{p}_{j_{k}}^{2}\lambda_{j_{1}}...\lambda_{j_{k}}^{k}V(\lambda_{j_{1}},...,\lambda_{j_{k}})\prod_{l=1}^{k}(\lambda_{j_{l}}-\lambda_{i}) \hat{p}_i^2\right]$$ 
$$=\sum_{j_1=1}^{r}\sum_{j_2=1}^r....\sum_{j_{k+1}=1}^r\hat{p}_{j_{1}}^{2}...\hat{p}_{j_{k+1}}^{2}\lambda_{j_{2}}...\lambda_{j_{k+1}}^{k}V(\lambda_{j_{1}},...,\lambda_{j_{k+1}})$$

$$=\sum_{(j_{1},..,j_{k+1})\in I^{+}_{k+1}}\left[ \hat{p}_{j_{1}}^{2}...\hat{p}_{j_{k+1}}^{2}V(\lambda_{j_{1}},...,\lambda_{j_{k+1}})\left( \sum_{\sigma \in S(1,..., k+1)} \epsilon(\sigma)\lambda_{j_{\sigma(2)}}\lambda_{j_{\sigma(k+1)}}^{k}\right) \right]$$
$$=\sum_{(j_{1},..,j_{k+1})\in I^{+}_{k+1}}\hat{p}_{j_{1}}^{2}...\hat{p}_{j_{k+1}}^{2}V(\lambda_{j_{1}},...,\lambda_{j_{k+1}})^{2}.$$
Therefore 
\begin{equation}
\label{eq:empirical-risk-bis}
\parallel Y-X\hat{\beta}_k\parallel^{2}=\dfrac{\sum_{(j_{1},..,j_{k+1})\in I^{+}_{k+1}}\hat{p}_{j_{1}}^{2}...\hat{p}_{j_{k+1}}^{2}V(\lambda_{j_{1}},...,\lambda_{j_{k+1}})^{2}}{\sum_{(j_{1},..,j_{k})\in I^{+}_{k}} \hat{p}_{j_{1}}^{2}...\hat{p}_{j_{k}}^{2}\lambda_{j_{1}}^{2}...\lambda_{j_{k}}^{2}V(\lambda_{j_{1}},...,\lambda_{j_{k}})^{2}}.
\end{equation}
Notice that another way to derive this result is to use the Gram matrix representation for orthogonal projections. Then, using the same arguments as the ones used to state Formula (\ref{eq:final-expression-residuals}), we can develop the two Gram determinants and get Equation (\ref{eq:empirical-risk-bis}).

Finally, using the fact that 
$$\sum_{(j_{1},..,j_{k+1})\in I^{+}_{k+1}}\hat{p}_{j_{1}}^{2}...\hat{p}_{j_{k+1}}^{2}V(\lambda_{j_{1}},...,\lambda_{j_{k+1}})^{2}=$$
$$\sum_{n>j_1>...>j_k\geq 1}\left[ \hat{p}_{j_{1}}^{2}...\hat{p}_{j_{k+1}}^{2}\lambda_{j_{1}}^{2},...,\lambda_{j_{k}}^{2}V(\lambda_{j_{1}},...,\lambda_{j_{k}})^{2}\sum_{i=j_1+1}^{r}\left( \prod_{l=1}^{k} \left(1-\frac{\lambda_i}{\lambda_{j_l}} \right)^2 \hat{p}_i^2\right)  \right],$$
we have proved Proposition \ref{prop:empirical_risk}.
\end{proof}

\subsection{Study of the empirical risk}
Now we have at hand an exact expression for the empirical risk. From this formula, we can easily provide a simplier and clearer upper bound for the PLS empirical risk, this is the objective of the next proposition.

\begin{prop}
\label{prop:upper_bound_empirical_risk}
Let $k<r$.
$$\parallel Y-X\hat{\beta}_k\parallel^{2}\leq \sum_{i=k+1}^{r}\left[  \prod_{l=1}^{k} \left(1-\frac{\lambda_i}{\lambda_{l}} \right)^2 \hat{p}_i^2\right] +
\left\lbrace 
\begin{array}{ccc}
 0 & \mbox{if} &  r=n \\
 \sum_{i=r+1}^{n}\hat{p}_i^2 & \mbox{if} & r<n
\end{array}\right ..$$
\end{prop}
Notice that if $\frac{\lambda_r}{\lambda_k}>1-\delta$ then $\sum_{i=k+1}^{r}\left[  \prod_{l=1}^{k} \left(1-\frac{\lambda_i}{\lambda_{l}} \right)^2\hat{p}_i^2\right]\leq \delta\sum_{i=k+1}^{r}\hat{p}_i^2 $.

\begin{proof}
This is a straightforward consequence of Proposition \ref{prop:empirical_risk} above because $$\sum_{r>j_1>...>j_k\geq 1}\left[\dfrac{ \hat{p}_{j_{1}}^{2}...\hat{p}_{j_{k+1}}^{2}\lambda_{j_{1}}^{2},...,\lambda_{j_{k}}^{2}V(\lambda_{j_{1}},...,\lambda_{j_{k}})^{2}}{\sum_{(j_{1},..,j_{k})\in I^{+}_{k}} \hat{p}_{j_{1}}^{2}...\hat{p}_{j_{k}}^{2}\lambda_{j_{1}}^{2}...\lambda_{j_{k}}^{2}V(\lambda_{j_{1}},...,\lambda_{j_{k}})^{2}}\right]\leq 1 $$
and therefore
$$\sum_{r>j_1>...>j_k\geq 1}\left[\dfrac{ \hat{p}_{j_{1}}^{2}...\hat{p}_{j_{k+1}}^{2}\lambda_{j_{1}}^{2},...,\lambda_{j_{k}}^{2}V(\lambda_{j_{1}},...,\lambda_{j_{k}})^{2}}{\sum_{(j_{1},..,j_{k})\in I^{+}_{k}} \hat{p}_{j_{1}}^{2}...\hat{p}_{j_{k}}^{2}\lambda_{j_{1}}^{2}...\lambda_{j_{k}}^{2}V(\lambda_{j_{1}},...,\lambda_{j_{k}})^{2}}\sum_{i=j_1+1}^{r} \prod \left(1-\frac{\lambda_i}{\lambda_{j_l}} \right)^2\hat{p}_i^2\right]  $$
$$ \leq\underset{I_k^+}{\textrm{max}}\left[ \sum_{i=j_1+1}^{r} \prod_{l=1}^{k} \left(1-\frac{\lambda_i}{\lambda_{j_l}} \right)^2\hat{p}_i^2\right] =\sum_{i=k+1}^{r} \left[ \prod_{l=1}^{k} \left(1-\frac{\lambda_i}{\lambda_{l}} \right)^2\hat{p}_i^2\right] .$$
\end{proof}

\begin{corollary}
\label{cor:ER}
Let $k<r$.
For a fixed design matrix we have
$$\mathbb{E}\left( \frac{1}{n}\parallel Y-X\hat{\beta}_k\parallel^{2}\right)$$
$$ \leq \frac{1}{n}\left(1-\frac{\lambda_n}{\lambda_1} \right)^{2k} \left[\sum_{i=k+1}^{r}\lambda_i\left( \beta^*_i \right) ^2+(r-k)\sigma^{2}\right]+
\left\lbrace 
\begin{array}{ccc}
 0 & \mbox{if} &  r=n \\
 \frac{1}{n}\sum_{i=r+1}^{n}\left( \lambda_i\left( \beta^*_i \right) ^2+\sigma^{2}\right) & \mbox{if} & r<n
\end{array}\right .. $$
\end{corollary}
For $n$ fixed the empirical risk decreases with an exponential rate in $k$.
Notice that the upper bound for the empirical risk stated in Corollary \ref{cor:ER} is tigher and more accurate than the one we have stated in a previous paper (see. \cite{BLAZERE}) using the min-max optimality of the Chebyschev polynomials.

In addition, from Proposition \ref{prop:empirical_risk}, it is obvious to show  that PLS fits closer than PCR:
\begin{corollary}
For $k\leq r$
$$\parallel Y-X\hat{\beta}_k\parallel^{2}\leq  \sum_{i=k+1}^{n}\hat{p}_i^2:=\parallel Y-X\hat{\beta}_{PCR}^k\parallel^{2}.$$ 
\end{corollary}
where by convention $ \sum_{i=
r+1}^{n}\hat{p}_i^2=0$ if $r= n$.
\begin{proof}
For all $i=k+1,...n$ 
$$0\leq\prod_{l=1}^{k} \left(1-\frac{\lambda_i}{\lambda_{l}} \right)\leq 1.$$
Therefore, $\sum_{i=k+1}^{r} \prod_{l=1}^{k} \left(1-\frac{\lambda_i}{\lambda_{l}} \right)^2\hat{p}_i^2\leq \sum_{i=k+1}^{r}\hat{p}_i^2.$
Then, we deduce from Proposition \ref{prop:upper_bound_empirical_risk} that 
$$\parallel Y-X\hat{\beta}_k\parallel^{2}\leq  \sum_{i=k+1}^{n}\hat{p}_i^2:=\parallel Y-X\hat{\beta}_{PCR}^k\parallel^{2}.$$
\end{proof}

In addition, because $$\parallel Y-X\hat{\beta}_k\parallel^{2}=\parallel Y-X\hat{\beta}_{OLS}\parallel^{2}+\parallel X\hat{\beta}_{OLS}-X\hat{\beta}_k\parallel^{2}= \sum_{i=r+1}^{n}\hat{p}_i^2+\parallel X\hat{\beta}_{OLS}-X\hat{\beta}_k\parallel^{2}$$
and $$\parallel Y-X\hat{\beta}_k\parallel^{2}\leq \sum_{i=k+1}^{n}\hat{p}_i^2 ,$$
we also conclude that $\parallel X\hat{\beta}_{OLS}-X\hat{\beta}_k\parallel^{2}\leq  \sum_{i=k+1}^{r}\hat{p}_i^2=\parallel X\hat{\beta}_{OLS}-X\hat{\beta}_{PCR}^{k}\parallel^{2}$.
This last result was proved earlier by \cite{JONG93}. A decade later \cite{PHATAK02} established a new proof of this result based on the connection between PLS and CG. Here, we have provided a very short proof of this particular feature of the PLS estimator.

\section{Mean Square Error}
\label{section:MSE}
In this section, we investigate the PLS Mean Square Error.
\subsection{Main result}
To evaluate the distance between the true and estimated parameter, a natural way consists in measuring the Mean Square Error (MSE) of the estimator defined by
$$MSE(\hat{\beta}_k):=\mathbb{E}\left[ \parallel X(\beta^*-\hat{\beta}_k)\parallel^{2} \right] $$
which is closely related to the prediction error.

Our main result is the following proposition.
\begin{prop}
\label{prop:MSE}
\begin{equation}
\label{eq:MSE}
\parallel X\beta^*-X\hat{\beta}_k\parallel^{2}=\sum_{i=1}^{r}\hat{Q}_k(\lambda_i)p_i^2+ \sum_{i=1}^{r}\left( 1-\hat{Q}_k(\lambda_i)\right)\varepsilon_i ^2.
\end{equation}
\end{prop}

\begin{proof}
We have $X\hat{\beta}_k=\sum_{i=1}^{r}\left(1- \hat{Q}_k(\lambda_i)\right)\hat{p}_iu_i$ (see Subsection \ref{subsection:interest}). Therefore
\begin{equation}
\label{eq:proofMSE}
 \parallel X\beta^*-X\hat{\beta}_k\parallel^{2}=\parallel \sum_{i=1}^{r}p_iu_i-\sum_{i=1}^{r}\left(1- \hat{Q}_k(\lambda_i)\right)\hat{p}_iu_i \parallel^{2}=\sum_{i=1}^{r}\left[ p_i-\left(1- \hat{Q}_k(\lambda_i)\right) \hat{p}_i\right]^{2}
 \end{equation}
 $$=\sum_{i=1}^{r}p_i^2-2\sum_{i=1}^{r}\left(1- \hat{Q}_k(\lambda_i)\right) p_i\hat{p}_i+\sum_{i=1}^{r}\left(1- \hat{Q}_k(\lambda_i)\right) \hat{p}_i^2$$
 because $\sum_{i=1}^{r}\left(1- \hat{Q}_k(\lambda_i)\right)^2 \hat{p}_i^2=\sum_{i=1}^{r}\left(1- \hat{Q}_k(\lambda_i)\right) \hat{p}_i^2$ (cf. Lemma \ref{lem:properties})
 $$=\sum_{i=1}^{r}p_i^2-\sum_{i=1}^{r}\left(1- \hat{Q}_k(\lambda_i)\right) \hat{p}_ip_i+\sum_{i=1}^{r}\left(1- \hat{Q}_k(\lambda_i)\right) \hat{p}_i\varepsilon_{i}$$
 using that $\hat{p}_i^2=\hat{p}_i(p_i+\varepsilon_{i})$
 $$=\sum_{i=1}^{r}p_i^2-\sum_{i=1}^{r}\left(1- \hat{Q}_k(\lambda_i)\right) p_i^2-\sum_{i=1}^{r}\left(1- \hat{Q}_k(\lambda_i)\right) p_i\varepsilon_i+\sum_{i=1}^{r}\left(1- \hat{Q}_k(\lambda_i)\right) p_i\varepsilon_{i}+\sum_{i=1}^{r}\left(1- \hat{Q}_k(\lambda_i)\right) \varepsilon_{i}^2$$
 $$=\sum_{i=1}^{r}\hat{Q}_k(\lambda_i)p_i^2+\sum_{i=1}^{r}\left( 1-\hat{Q}_k(\lambda_i)\right)\varepsilon_i ^2.$$
 Notice that, in Equation (\ref{eq:proofMSE}), if we first write that $$\parallel \sum_{i=1}^{r}p_iu_i-\sum_{i=1}^{r}\left(1- \hat{Q}_k(\lambda_i)\right)\hat{p}_iu_i \parallel^{2}=\sum_{i=1}^{r}\left[ \hat{Q}_k(\lambda_i)p_i-\left(1-\hat{Q}_k(\lambda_i)\right) \varepsilon_i\right] ^{2}$$ and then expand the square we would not have been able to derive an expression for the MSE as simple as the one established in Proposition \ref{prop:MSE}.
 \end{proof}
 The reals $\left(\hat{Q}_k(\lambda_i) \right)_{1\leq i\leq r} $ are random so that we cannot establish, from Proposition \ref{prop:MSE}, a classical bias-variance decomposition for the PLS estimator. In fact, in the light of the expression of the residual polynomials (cf. Theorem \ref{theo: expression-det-2}), it seems quite difficult and even infeasible to compute the variance of $\hat{\beta}_k$ or the one of $X\hat{\beta}_k$. Indeed, the variances along the eigenvectors directions are not obvious and even not mutually independant.
But we can compare Formula (\ref{eq:MSE}) to the bias-variance decomposition obtained in the case of a shrinkage estimator with deterministic filter factors.
Actually, it is well known that for an estimator $\hat{\beta}_{S}$ of $\beta^*$ of the form
$$\hat{\beta}_{S}=\sum_{i=1}^{r}f(\lambda_i)\frac{\hat{p}_i}{\sqrt{\lambda_i}}u_i $$
we have 
\begin{equation}
\label{eq:derteministic}
MSE(\hat{\beta}_{S})=\sum_{i=1}^{r}\left(1-f(\lambda_i) \right) ^{2}p_i^2+ \sigma^{2}\sum_{i=1}^{r}\left( f(\lambda_i)\right)^2.
\end{equation}

Therefore, recording that $f_i^{(k)}:=1-\hat{Q}_k(\lambda_i)$ are the filter factors of the PLS estimator (see Section \ref{section:filter}), we can obtain from Proposition \ref{prop:MSE} an expression similar to (\ref{eq:derteministic}):
\begin{equation}
\label{eq:random}
 \parallel X\beta^*-\hat{\beta}_k\parallel^{2}=\sum_{i=1}^{r}\left( 1-f_i^{(k)}\right)p_i^2+ \sum_{i=1}^{r}f_i^{(k)}\varepsilon_i ^2
 \end{equation}

and 
$$
MSE(\hat{\beta}_k)=\sum_{i=1}^{r}\mathbb{E}\left[ 1-f_i^{(k)} \right] p_i^2+ \sum_{i=1}^{r}\mathbb{E}\left[f_i^{(k)}\varepsilon_i ^2 \right] .
$$
However, this is not a classical bias variance decomposition. 
Of course, in the deterministic case, a filter factor larger than one always increases the MSE compared to the one of the OLS because this implies an increase of both the bias and the variance (see Equation (\ref{eq:derteministic})). Because the PLS filter factors can be larger than one, \cite{FRANCK93} proposed to bound by one the absolute value of the PLS shrinkage factors larger than one. In other words they propose to define
$$\tilde{f}_i^{(k)}=\left\lbrace \begin{array}{ll}
+1& \mbox{if}~ f_i^{(k)}>+1\\
-1&  \mbox{if}~ f_i^{(k)}<-1\\
f_i^{(k)}& \mbox{otherwise}
\end{array}\right.$$
and $\tilde{\beta}_k=\sum_{i=1}^{r}\tilde{f}_i^{(k)}\dfrac{\hat{p}_i}{\sqrt{\lambda_i}}v_i$ as a new estimator of $\beta^{*}$ derived from the PLS one.
 However, as pointed out by \cite{KRA07}, a filter factor larger than one in the case of PLS does not necessarily imply a larger MSE. So bounding the absolute value of the PLS shrinkage factors by one not always lead to a better MSE. She precises that it is not clear why we have such a peculiar behaviour of the PLS filter factor. However she illustrates this point through simulations. Here, having a look at Equation (\ref{eq:random}), we better understand why such a particular behaviour of the PLS estimator. Actually, from Equation (\ref{eq:random}), we see that the filter factors or their difference to one are not squared, contrary to what happens in the case of deterministic filter factors (see Equation (\ref{eq:derteministic})). So a filter factor $f_i^{(k)}$ larger than one does not necessarily increase the MSE because it can be balanced by $\left( 1-f_i^{(k)}\right)p_i^2$ which in this case will be negative.

\subsection{Another decomposition of the MSE}
\subsubsection{Decomposition of the MSE through projection onto Krylov subspaces}
Let $\hat{\Pi}_{k}$ be the orthogonal projector onto the space spanned by $(XX^T)Y,...,(XX^T)^kY$. Because $X\hat{\beta}_k=\hat{\Pi}_{k}Y$, we have
$$ \parallel X\beta^*-X\hat{\beta}_k\parallel^{2}=\parallel X\beta^*-\hat{\Pi}_{k}X\beta^*\parallel^{2}+\parallel \hat{\Pi}_{k}\varepsilon\parallel^{2}.$$

\begin{lem}
\label{lem:decompo1}
\begin{equation}
\label{lem:eq1}
 \parallel X\beta^*-\hat{\Pi}_{k}X\beta^*\parallel^{2}=\sum_{i=1}^{r}\hat{Q}_k(\lambda_i)
\hat{p}_ip_i -\frac{1}{C}\left| \begin{array}{cccc}
\sum_{j=1}^{r}\varepsilon_{j}p_j& \sum_{j=1}^{r}\lambda_{j}p_j\hat{p}_{j} &...  & \sum_{j=1}^{r}\lambda_{j}^{k}p_j\hat{p}_{j} \\ 
\sum_{j=1}^{r}\lambda_{j}\varepsilon_{j}\hat{p}_{j}& \sum_{j=1}^{r}\lambda_{j}^{2}\hat{p}_{j}&  & \sum_{j=1}^{r}\lambda_{j}^{k+1}\hat{p}_{j}^{2} \\
 \vdots &  &  &  \\  
\sum_{j=1}^{r}\lambda_{j}^{k}\varepsilon_{j}\hat{p}_{j}& \sum_{j=1}^{r}\lambda_{j}^{k+1}\hat{p}_{j}^{2}&  & \sum_{j=1}^{r}\lambda_{j}^{2k}\hat{p}_{j}^{2}
\end{array} \right|,
\end{equation}
where $C:=\sum_{(j_{1},..,j_{k})\in I^{+}_{k}} \hat{p}_{j_{1}}^{2}...\hat{p}_{j_{k}}^{2}\lambda_{j_{1}}^{2}...\lambda_{j_{k}}^{2}V(\lambda_{j_{1}},...,\lambda_{j_{k}})^{2}$.
\end{lem}

Having a look to Equation (\ref{lem:eq1}), we see that there is no hope to provide a simple expression for the expectation of $\parallel X\beta^*-\hat{\Pi}_{k}X\beta^*\parallel^{2}$. However, Lemma \ref{lem:decompo2} below show that the expression of $\parallel X\beta^*-\hat{\Pi}_{k}X\beta^*\parallel^{2}$ can simplify in part with the one of $\parallel \hat{\Pi}_{k}\varepsilon\parallel^{2}$.

\begin{lem}
\label{lem:decompo2}
\begin{equation}
\parallel \hat{\Pi}_{k}\varepsilon\parallel^{2}= \sum_{i=1}^{r}\varepsilon_{i}^2+\sum_{i=1}^{r}\hat{Q}_k(\lambda_i)\hat{p}_i\varepsilon_i-\frac{1}{C}\left| \begin{array}{cccc}
\sum_{j=1}^{r}\varepsilon_{j}p_j& \sum_{j=1}^{r}\lambda_{j}\varepsilon_j\hat{p}_{j} &...  & \sum_{j=1}^{r}\lambda_{j}^{k}\varepsilon_j\hat{p}_{j} \\ 
\sum_{j=1}^{r}\lambda_{j}p_{j}\hat{p}_{j}& \sum_{j=1}^{r}\lambda_{j}^{2}\hat{p}_{j}&  & \sum_{j=1}^{r}\lambda_{j}^{k+1}\hat{p}_{j}^{2} \\
 \vdots &  &  &  \\  
\sum_{j=1}^{r}\lambda_{j}^{k}p_{j}\hat{p}_{j}& \sum_{j=1}^{r}\lambda_{j}^{k+1}\hat{p}_{j}^{2}&  & \sum_{j=1}^{r}\lambda_{j}^{2k}\hat{p}_{j}^{2}
\end{array} \right|,
\end{equation}
where we recall that $C:=\sum_{(j_{1},..,j_{k})\in I^{+}_{k}} \hat{p}_{j_{1}}^{2}...\hat{p}_{j_{k}}^{2}\lambda_{j_{1}}^{2}...\lambda_{j_{k}}^{2}V(\lambda_{j_{1}},...,\lambda_{j_{k}})^{2}.$
\end{lem}
\begin{proof}
$$\parallel \hat{\Pi}_{k}\varepsilon\parallel^{2}=\parallel \varepsilon\parallel^{2}-\parallel \varepsilon- \hat{\Pi}_{k}\varepsilon\parallel^{2}.$$
Then, using similar arguments as those used to prove Lemma \ref{lem:decompo1} (by reversing the role played by $p_i$ and $\varepsilon_i$), we get $$\parallel \varepsilon- \hat{\Pi}_{k}\varepsilon\parallel^{2}=$$
$$ \sum_{i=r+1}^{n}\varepsilon_{i}^2+\sum_{i=1}^{r}\hat{Q}_k(\lambda_i)\hat{p}_i\varepsilon_i-\frac{1}{C}\left| \begin{array}{cccc}
\sum_{j=1}^{r}\varepsilon_{j}p_j& \sum_{j=1}^{r}\lambda_{j}\varepsilon_j\hat{p}_{j} &...  & \sum_{j=1}^{r}\lambda_{j}^{k}\varepsilon_j\hat{p}_{j} \\ 
\sum_{j=1}^{r}\lambda_{j}p_{j}\hat{p}_{j}& \sum_{j=1}^{r}\lambda_{j}^{2}\hat{p}_{j}&  & \sum_{j=1}^{r}\lambda_{j}^{k+1}\hat{p}_{j}^{2} \\
 \vdots &  &  &  \\  
\sum_{j=1}^{r}\lambda_{j}^{k}p_{j}\hat{p}_{j}& \sum_{j=1}^{r}\lambda_{j}^{k+1}\hat{p}_{j}^{2}&  & \sum_{j=1}^{r}\lambda_{j}^{2k}\hat{p}_{j}^{2}
\end{array} \right|.$$
\end{proof}

Here again, we see that the expression of $\parallel \hat{\Pi}_{k}\varepsilon\parallel^{2}$ depends in an intricated way of the noise. So that it is not feasible to calculate its expectation. However, from Lemma \ref{lem:decompo1} and \ref{lem:decompo2}, we deduce
\begin{prop}
\begin{equation}
\label{eq:MSE2}
\parallel X\beta^*-X\hat{\beta}_k\parallel^{2}=\sum_{i=1}^{r}\hat{Q}_k(\lambda_i)
\hat{p}_ip_i+\sum_{i=1}^{r}\varepsilon_{i}^2-\sum_{i=1}^{r}\hat{Q}_k(\lambda_i)\hat{p}_i\varepsilon_i.\end{equation}
\end{prop}
Notice that the above expression of the MSE is equivalent to the one stated in Proposition \ref{prop:MSE}. Actually, we could have easily deduce Equation (\ref{eq:MSE2})
from Equation (\ref{eq:MSE}). But, it appears that it was also informative to consider the decomposition through the orthogonal projector onto the Krylov subspaces. Indeed, this shows that the classical decomposition (when dealing with projection onto fixed subspaces) is in the case of PLS much more complicated. Hopefully, simplifications based on the intrinsic properties of this estimator lead to a decomposition easier to study.

\subsubsection{Proof of Lemma \ref{lem:decompo1}.}
\begin{proof}
We can express $ \parallel X\beta^*-\hat{\Pi}_{k}X\beta^*\parallel^{2}$ in terms of the ratio of two Gram determinants:
$$\parallel X\beta^*-\hat{\Pi}_{k}X\beta^*\parallel^{2}=\dfrac{G_1}{G_2}$$
where $$G_1=\left| \begin{array}{cccc}
\left\langle X\beta^*,  X\beta^* \right\rangle  & \left\langle X\beta^*,  XX^TY \right\rangle  &...  & \left\langle X\beta^*,  (XX^T)^kY \right\rangle \\  
 \left\langle X\beta^*,  XX^TY \right\rangle & \left\langle XX^TY ,  XX^TY \right\rangle &  & \left\langle XX^TY ,  (XX^T)^kY \right\rangle \\ 
 \vdots &  &  &  \\
\left\langle X\beta^* ,  (XX^T)^kY \right\rangle & \left\langle XX^TY ,  (XX^T)^kY \right\rangle& ... & \left\langle (XX^T)^{k}Y ,  (XX^T)^kY \right\rangle
\end{array} \right|$$

and 
$$
G_2=\left| \begin{array}{cccc}
\left\langle XX^TY,  XX^TY \right\rangle & \left\langle XX^TY,  (XX^T)^2Y \right\rangle&...  & \left\langle XX^TY,  (XX^T)^kY \right\rangle \\  
 \left\langle (XX^T)^2Y,  XX^TY \right\rangle & \left\langle (XX^T)^2Y,  (XX^T)^2Y \right\rangle&...  & \left\langle (XX^T)^2Y,  (XX^T)^kY \right\rangle \\ 
 \vdots &  &  &  \\
\left\langle (XX^T)^kY ,  XX^TY \right\rangle & \left\langle  (XX^T)^kY,  (XX^T)^2Y \right\rangle& ... & \left\langle (XX^T)^kY ,  (XX^T)^kY \right\rangle
\end{array} \right|.$$
On one hand, using the eigendecomposition of $X$, we have 
$$G_2=
\left| \begin{array}{cccc}
\sum_{j=1}^{r}\lambda_{j}^2\hat{p}_{j}^{2} & \sum_{j=1}^{r}\lambda_{j}^{3}\hat{p}_{j}^{2} &...  & \sum_{j=1}^{r}\lambda_{j}^{k+1}\hat{p}_{j} ^{2}\\ 
\sum_{j=1}^{r}\lambda_{j}^{3}\hat{p}_{j}^{2}& \sum_{j=1}^{r}\lambda_{j}^{4}\hat{p}_{j}^{2}&  & \sum_{j=1}^{r}\lambda_{j}^{k+2}\hat{p}_{j}^{2} \\
 \vdots &  &  &  \\  
\sum_{j=1}^{r}\lambda_{j}^{k+1}\hat{p}_{j}^{2}& \sum_{j=1}^{r}\lambda_{j}^{k+2}\hat{p}_{j}^{2}&  & \sum_{j=1}^{r}\lambda_{j}^{2k}\hat{p}_{j}^{2}
\end{array} \right|= \left| \begin{array}{cccc}
\hat{m}_{1} & \hat{m}_{2} & ... & \hat{m}_{k} \\ 
 \vdots &  &  &  \\ 
\hat{m}_{k-1} & \hat{m}_{k} &  & \hat{m}_{2k-2} \\ 
 \hat{m}_{k}&  \hat{m}_{k+1} & ... &  \hat{m}_{2k-1}
\end{array} \right|$$
where  $\hat{m}_{i}=\int x^{i}\hat{\mu}$.
Therefore 
\begin{equation}
\label{eq:G2}
G_2=\sum_{(j_{1},..,j_{k})\in I^{+}_{k}} \hat{p}_{j_{1}}^{2}...\hat{p}_{j_{k}}^{2}\lambda_{j_{1}}^{2}...\lambda_{j_{k}}^{2}V(\lambda_{j_{1}},...,\lambda_{j_{k}})^{2}.
\end{equation} 
(see one of the argument used in the proof of Theorem 4.1 in \cite{BLAZERE}).

On the other hand, we have
$$G_1=
\left| \begin{array}{cccc}
\sum_{j=1}^{r}p_j^2 & \sum_{j=1}^{r}\lambda_{j}p_j\hat{p}_{j} &...  & \sum_{j=1}^{r}\lambda_{j}^{k}p_j\hat{p}_{j} \\ 
\sum_{j=1}^{r}\lambda_{j}p_j\hat{p}_{j}& \sum_{j=1}^{r}\lambda_{j}^{2}\hat{p}_{j}&  & \sum_{j=1}^{r}\lambda_{j}^{k+1}\hat{p}_{j}^{2} \\
 \vdots &  &  &  \\  
\sum_{j=1}^{r}\lambda_{j}^{k}p_j\hat{p}_{j}& \sum_{j=1}^{r}\lambda_{j}^{k+1}\hat{p}_{j}^{2}&  & \sum_{j=1}^{r}\lambda_{j}^{2k}\hat{p}_{j}^{2}
\end{array} \right|$$
$$= \sum_{i=1}^{r}\sum_{j_{1}=1}^{r}...\sum_{j_k=1}^{r}p_i\hat{p}_{j_1}...\hat{p}_{j_k}\left| \begin{array}{cccc}
p_i & \lambda_{i}\hat{p}_{i} &...  & \lambda_{i}^{k}\hat{p}_{i}\\ 
\lambda_{j_1}p_{j_1}& \lambda_{j_1}^{2}\hat{p}_{j_1}&  & \lambda_{j_1}^{k+1}\hat{p}_{j_1} \\
 \vdots &  &  &  \\  
\lambda_{j_k}^{k}p_{j_k}& \lambda_{j_k}^{k+1}\hat{p}_{j_k}&  & \lambda_{j_k}^{2k}\hat{p}_{j_k}
\end{array} \right| $$
Then, using $\hat{p}_i=p_i+\varepsilon_i$, we get
$$G_1= \sum_{i=1}^{r}\sum_{j_{1}=1}^{r}...\sum_{j_k=1}^{r}p_i\hat{p}_{j_1}...\hat{p}_{j_k}\left| \begin{array}{cccc}
\hat{p}_i & \lambda_{i}\hat{p}_{i} &...  & \lambda_{i}^{k}\hat{p}_{i}\\ 
\lambda_{j_1}\hat{p}_{j_1}& \lambda_{j_1}^{2}\hat{p}_{j_1}&  & \lambda_{j_1}^{k+1}\hat{p}_{j_1} \\
 \vdots &  &  &  \\  
\lambda_{j_k}^{k}\hat{p}_{j_k}& \lambda_{j_k}^{k+1}\hat{p}_{j_k}&  & \lambda_{j_k}^{2k}\hat{p}_{j_k}
\end{array} \right|$$
$$- \sum_{i=1}^{r}\sum_{j_{1}=1}^{r}...\sum_{j_k=1}^{r}p_i\hat{p}_{j_1}...\hat{p}_{j_k}\left| \begin{array}{cccc}
\varepsilon_i & \lambda_{i}\hat{p}_{i} &...  & \lambda_{i}^{k}\hat{p}_{i}\\ 
\lambda_{j_1}\varepsilon_{j_1}& \lambda_{j_1}^{2}\hat{p}_{j_1}&  & \lambda_{j_1}^{k+1}\hat{p}_{j_1} \\
 \vdots &  &  &  \\  
\lambda_{j_k}^{k}\varepsilon_{j_k}& \lambda_{j_k}^{k+1}\hat{p}_{j_k}&  & \lambda_{j_k}^{2k}\hat{p}_{j_k}
\end{array} \right| $$

$$=\sum_{i=1}^{r}p_i\hat{p}_i \sum_{j_{1}=1}^{r}...\sum_{j_k=1}^{r}\hat{p}_{j_1}^2...\hat{p}_{j_k}^{2}\left| \begin{array}{cccc}
 1& \lambda_{i} &...  & \lambda_{i}^{k}\\ 
\lambda_{j_1}& \lambda_{j_1}^{2}&  & \lambda_{j_1}^{k+1} \\
 \vdots &  &  &  \\  
\lambda_{j_k}^{k}& \lambda_{j_k}^{k+1}&  & \lambda_{j_k}^{2k}
\end{array} \right | $$
$$-\left| \begin{array}{cccc}
\sum_{j=1}^{r}\varepsilon_{j}p_j& \sum_{j=1}^{r}\lambda_{j}p_j\hat{p}_{j} &...  & \sum_{j=1}^{r}\lambda_{j}^{k}p_j\hat{p}_{j} \\ 
\sum_{j=1}^{r}\lambda_{j}\varepsilon_{j}\hat{p}_{j}& \sum_{j=1}^{r}\lambda_{j}^{2}\hat{p}_{j}&  & \sum_{j=1}^{r}\lambda_{j}^{k+1}\hat{p}_{j}^{2} \\
 \vdots &  &  &  \\  
\sum_{j=1}^{r}\lambda_{j}^{k}\varepsilon_{j}\hat{p}_{j}& \sum_{j=1}^{r}\lambda_{j}^{k+1}\hat{p}_{j}^{2}&  & \sum_{j=1}^{r}\lambda_{j}^{2k}\hat{p}_{j}^{2}
\end{array} \right|$$
$$=\sum_{i=1}^{r}p_i\hat{p}_i\left[ \sum_{(j_{1},..,j_{k})\in I^{+}_{k}} \hat{p}_{j_{1}}^{2}...\hat{p}_{j_{k}}^{2}\lambda_{j_{1}}^{2}...\lambda_{j_{k}}^{2}V(\lambda_{j_{1}},...,\lambda_{j_{k}})^{2} \prod_{l=1}^{k}(1-\frac{x}{\lambda_{j_{l}}})\right] $$
\begin{equation}
\label{eq:G1}
-\left| \begin{array}{cccc}
\sum_{j=1}^{r}\varepsilon_{j}p_j& \sum_{j=1}^{r}\lambda_{j}p_j\hat{p}_{j} &...  & \sum_{j=1}^{r}\lambda_{j}^{k}p_j\hat{p}_{j} \\ 
\sum_{j=1}^{r}\lambda_{j}\varepsilon_{j}\hat{p}_{j}& \sum_{j=1}^{r}\lambda_{j}^{2}\hat{p}_{j}&  & \sum_{j=1}^{r}\lambda_{j}^{k+1}\hat{p}_{j}^{2} \\
 \vdots &  &  &  \\  
\sum_{j=1}^{r}\lambda_{j}^{k}\varepsilon_{j}\hat{p}_{j}& \sum_{j=1}^{r}\lambda_{j}^{k+1}\hat{p}_{j}^{2}&  & \sum_{j=1}^{r}\lambda_{j}^{2k}\hat{p}_{j}^{2}
\end{array} \right|
\end{equation}
 because $$ \sum_{j_{1}=1}^{r}...\sum_{j_k=1}^{r}\hat{p}_{j_1}^2...\hat{p}_{j_k}^{2}\left| \begin{array}{cccc}
 1& \lambda_{i} &...  & \lambda_{i}^{k}\\ 
\lambda_{j_1}& \lambda_{j_1}^{2}&  & \lambda_{j_1}^{k+1} \\
 \vdots &  &  &  \\  
\lambda_{j_k}^{k}& \lambda_{j_k}^{k+1}&  & \lambda_{j_k}^{2k}
\end{array} \right|$$
$$=\sum_{(j_{1},..,j_{k})\in I^{+}_{k}}\left[  \hat{p}_{j_{1}}^{2}...\hat{p}_{j_{k}}^{2}\lambda_{j_{1}}^{2}...\lambda_{j_{k}}^{2}V(\lambda_{j_{1}},...,\lambda_{j_{k}})^{2} \prod_{l=1}^{k}(1-\frac{\lambda_i}{\lambda_{j_{l}}})\right]. $$
 
From Equation (\ref{eq:G1}), (\ref{eq:G2}) and the expression of $\hat{Q}_k$ (cf. Theorem \ref{theo: expression-det-2}) we get Lemma \ref{lem:decompo1}.
\end{proof}

\section{Conclusion}
In this paper, through the expression obtained for the residuals, we have proposed a new approach for the PLS method. We have established new exact expressions for the main PLS objects (filter factors, estimator, empirical risk, MSE). This is useful to provide new interpretations and to shed new light on the behaviour of PLS. Furthermore, this approach provides a unified framework to recover well known properties of the PLS estimator proved earlier through different methods by De Jong, Goutis, Lingjaerde and Christophersen, Phatak and de Hoog or Kramer. Our approach is powerful as it allows both to recover all these properties at the same time and new ones.
\bibliographystyle{apalike}
\bibliography{PLSbiblio}

\end{document}